\documentclass{svjour3}     

\usepackage{graphicx} 					
\usepackage{subcaption}					
\usepackage{amsmath} 					
\usepackage{amssymb} 					
\usepackage{bm} 						
\usepackage{spalign} 					
\usepackage{mathtools} 					
\usepackage{algorithm} 					
\usepackage{algpseudocode} 				
\usepackage{booktabs}					
\usepackage{enumitem}       			
\usepackage[pagewise]{lineno} 			
\usepackage{numprint}					
\usepackage{url}						
\usepackage{mathptmx}      				
\usepackage{xcolor}                     

\setenumerate[1]{label={(\arabic*)}} 	
\npthousandsep{\,}						

\smartqed  								

\newcommand{\Th}[0]{\mathcal{T}}

\newcommand{\Zh}[0]{\mathcal{Z}}
\newcommand{\Bh}[0]{\mathcal{B}}
\newcommand{\Fh}[0]{\mathcal{F}}

\newcommand{\R}[0]{\mathbb{R}}
\newcommand{\N}[0]{\mathbb{N}}

\let\mat=\spalignmat

\let\vec=\spalignvector

\newcommand{\vertiii}[1]{{\left\vert\kern-0.25ex\left\vert\kern-0.25ex\left\vert #1 \right\vert\kern-0.25ex\right\vert\kern-0.25ex\right\vert}}
\newcommand{\norm}[2][]{\| #2 \|_{#1}}
\newcommand{\normbig}[2][]{\left\|#2\right\|_{#1}}

\providecommand\given{} 
\newcommand\SetSymbol[1][]{
   \nonscript\,#1: \allowbreak \nonscript\,\mathopen{}}
\DeclarePairedDelimiterX\Set[1]{\lbrace}{\rbrace}%
 { \renewcommand\given{\SetSymbol[\delimsize]} #1 }

\newcommand{\expnumber}[2]{{#1}\mathrm{e}{#2}}

\newtheorem{assumption}{Assumption}

\journalname{BIT}

\begin{document}

\title{Hybrid Nitsche method for distributed computing\thanks{This work was supported by the Research Council of Finland (Flagship of Advanced Mathematics for Sensing Imaging and Modelling grant 359181) and the Portuguese government through FCT (Funda\c c\~ao para a Ci\^encia e a Tecnologia), I.P., under the project UIDB/04459/2025.}}

\author{Tom Gustafsson    \and
		Antti Hannukainen \and
		Vili Kohonen 	  \and
		Juha Videman
}

\institute{T. Gustafsson \at
Department of Mechanical Engineering \\
			Aalto University \\
            P.O. Box 11100 00076 Aalto \\ 
            Espoo, Finland\\
              \email{tom.gustafsson@aalto.fi}           
           \and
           A. Hannukainen \at
Department of Mathematics and Systems Analysis \\
Aalto University \\ 
	\email{antti.hannukainen@aalto.fi}
			\and 
			V. Kohonen \at 
Department of Mathematics and Systems Analysis \\
Aalto University \\ 
	\email{vili.kohonen@aalto.fi}
			\and 
			J. Videman \at 
			CAMGSD/Departamento de Matematica \\
			Instituto Superior Tecnico, Universidade de Lisboa \\
            Av. Rovisco Pais 1, 1049-001 \\ 
            Lisbon, Portugal \\
			\email{jvideman@math.tecnico.ulisboa.pt}
}







\date{}

\maketitle

\begin{abstract}
		We extend a distributed finite element method built upon model order reduction to arbitrary polynomial degree using a hybrid Nitsche scheme. The new method considerably simplifies the transformation of the finite element system to the reduced basis for large problems. 
		We prove that the error of the reduced Nitsche solution converges optimally with respect to the approximation order of the finite element spaces and linearly with respect to the dimension reduction parameter. 
		Numerical tests with nontrivial tetrahedral meshes using second-degree polynomial bases support the theoretical results.
\end{abstract}

\section{Introduction}
\label{sec:intro}




Iterative substructuring methods \cite{tosellidomain2005,farhat2001feti,dohrmann2003preconditioner}, or Schur complement methods, are highly scalable in parallel finite element computations.  They rely on a massively parallel supercomputer with fast interconnect. The Schur complement interface system is constructed subdomain-wise in several compute nodes, and the nodes need to communicate during the preconditioned conjugate gradient iteration.
On the other hand, 
if the interface problem can be made sufficiently small and sparse through independent local procedures on each subdomain, the problem could be solved on a single large memory node. This would eliminate the need for internode communication and could take advantage of independent compute nodes, widely and flexibly available in the cloud.

Optimal local approximation spaces have been used to significantly reduce the dimensionality of the approximations to PDEs in multiscale modelling \cite{babuskaoptimal2011,babuskamachine2014,babuskamultiscalespectral2020}. In essence,
each subproblem is used to find a local reduced basis and the original problem is projected to a small subspace crafted from the local bases. Several works improved the computational efficiency of forming the local reduced bases using randomized numerical linear algebra \cite{calorandomized2016,buhrrandomized2018,schleusoptimal2022}, albeit testing only small to moderately-sized problems. For large-scale problems, a Laplace eigenvalue problem of 10M degrees-of-freedom was computed in a distributed setting with a similar approach \cite{hannukainendistributed2022}. In \cite{gustafsson2024}, 
we presented an efficient novel methodology to compute the local bases and project the global system to the reduced one and, using only a laptop and the cloud, ran numerical examples with linear finite elements up to 85M degrees-of-freedom for the Poisson problem.

The methodology in \cite{gustafsson2024} 
employed a first-degree polynomial basis, which simplified the combination of the local bases. In this paper, we generalize this approach to incorporate arbitrary degree polynomial bases using a hybrid Nitsche scheme \cite{egger2009class,egger2012hybrid,burman2019hybridized,hansbo2022nitsche}. 
Hybrid Nitsche is a mortar method that ensures continuity over the subdomain interfaces by introducing a trace variable. The original system is handled subdomain-wise with the trace variable, see Section 2 for the analytical setting and Section 5 for the FEM implementation combined with the model order reduction. The problem thus becomes similar to the interface problem for conventional iterative substructuring methods.

The hybrid Nitsche scheme also streamlines the implementation compared to \cite{gustafsson2024}. The trace variable simplifies the local problems as no extra layer of overlapping elements is required and works identically for any polynomial basis, while partition-of-unity-based approaches require constructing a cutoff function and specialized assembly routines. Further, the reduced linear system of equations is acquired completely from the local operations and neither global assembly nor transformations are required, compared to the careful indexing and global projection that was required in our previous work. While the block-diagonal stiffness matrix could be inverted block-by-block without local model order reduction, this would require a supercomputer environment due to memory requirements. The local reduced bases transform the subdomain stiffness matrices into small diagonal matrices. The Schur complement system then fits into laptop memory even for large original problems. A simple diagonal preconditioner is used to solve the reduced system with the preconditioned conjugate gradient method.

In addition to permitting arbitrary polynomial degree and significantly simplifying the computations, the hybrid Nitsche approach avoids the (nominal) $h^{-1}$ scaling in the partition-of-unity-style reduction error estimate of \cite{gustafsson2024}. The method retains polynomial convergence in the mesh parameter $h$ and now converges linearly with respect to the user-specified dimension reduction tolerance $\epsilon$. The theoretical estimate is supported by numerical examples using second-degree polynomial basis. With small enough tolerance $\epsilon$, the method produces practically identical results to FEM. 

A larger example of 20M degrees-of-freedom, performed on a laptop and the cloud on a challenging domain for the method, shows the method's applicability for large-scale computing. 
Multigrid methods offer an alternative efficient approach for solving large problems on single compute nodes, especially when problem geometries exhibit natural hierarchies and material parameters are simple \cite{briggs2000multigrid,trottenberg2001multigrid,falgout2006introduction}. On the other hand, 
our domain decomposition approach maintains consistent performance across complex geometries and heterogeneous materials without the need for problem-specific tuning. However, the performance for our method is still bounded because we have not constructed a specialized preconditioner for the interface problem, a necessary development for fully utilizing the method. 

The structure of the paper is as follows. We derive the hybrid Nitsche scheme from the Lagrange multiplier approach to the interface continuity problem in Section \ref{sec:problem}. Then, the local model order reduction methodology is constructed in Section \ref{sec:localmor}. Section \ref{sec:error} is used to prove an error estimate for the model order reduction. 
Section \ref{sec:implementation} shortly describes the FEM implementation, and in Section \ref{sec:numerics} we provide numerical experiments to validate the theoretical claims. Conclusions in Section \ref{sec:conclusions} close the paper.

\section{Problem formulation}
\label{sec:problem}

\subsection{Continuous formulation}

Let $\Omega\subset\R^d, d\in \Set{2,3}$ be a polygonal/polyhedral domain and $\Set{\Omega_i}_{i=1}^n$ its partition into $n$ subdomains. Define the skeleton $\Gamma = \cup_{i=1}^n \partial\Omega_i$ and the trace space 
$V_0 := H^1_0(\Omega)|_\Gamma$.

Assume $f\in L^2(\Omega)$. We examine the minimization problem 
\begin{align}
	\label{eq:energy}
	\inf_{u\in H_0^1(\Omega)} J(u) = \inf_{u\in H_0^1(\Omega)} \frac{1}{2}(\nabla u, \nabla u)_\Omega - (f,u)_\Omega.
\end{align}
For domain decomposition,
we rewrite the functional $J: H^1_0(\Omega)\to \R$ by splitting it with respect to the partition and adding a term to ensure the continuity of the solution across (or up to) the skeleton $\Gamma$. Let $u_0\in V_0$ be the trace variable defined on the skeleton $\Gamma$ and $u_i\in H^1(\Omega_i)$ the solution on $\Omega_i$. Let $u = (u_1, \dots, u_n, u_0)$ and denote by $W := \Pi_{i=1}^n H^1(\Omega_i) \times V_0$ for the solution space. Moreover, define over each $\partial\Omega_i, i=1,\dots, n$ the dual space $\varLambda_i = (H^1(\Omega_i)|_{\partial\Omega_i})'$, and let $\varLambda := \Pi_{i=1}^n \varLambda_i.$ Consider the following saddle point problem:
\begin{align}
	\label{eq:energyaug}
	\inf_{\substack{u_i\in H^1(\Omega_i)\\ u_0\in V_0}} \sup_{\lambda_i\in \varLambda_i} \sum_{i=1}^n \left(\frac{1}{2}(\nabla u_i, \nabla u_i)_{\Omega_i} - (f, u_i)_{\Omega_i} - \langle\lambda_i, u_i - u_0\rangle\right),
\end{align}
where 
$\lambda_i$ are Lagrange multipliers associated with the constraints $u_i=u_0$ on $\partial\Omega_i$, and where by $\langle\cdot, \cdot\rangle: \varLambda_i \times H^1(\Omega_i)|_{\partial\Omega_i}\to \R$ we denote the duality pairing.  Equivalence of problems \eqref{eq:energy} and \eqref{eq:energyaug} can be deduced from the general saddle point theory described, e.g., in \cite{boffi2013mixed}.

The variational formulation of problem \eqref{eq:energyaug} reads as follows:
find $(u,\lambda) \in W\times \varLambda$ such that for all $(v,\mu)\in W\times\varLambda$
\begin{equation}
	\begin{aligned}
		\label{eq:energyvariational}
		\sum_{i=1}^n \bigg( (\nabla u_i, \nabla v_i)_{\Omega_i} - \langle\lambda_i, v_i - v_0\rangle - \langle\mu_i, u_i - u_0\rangle \bigg) = \sum_{i=1}^n (f, v_i)_{\Omega_i} \\
	\end{aligned}
\end{equation}
holds.
\begin{remark}
	The strong form of \eqref{eq:energyvariational} is 
	\begin{equation}
		\begin{aligned}
			\label{eq:strongform}
			-\Delta u_i &= f \quad && \text{in $\Omega_i$},\\
			u_i &= u_0 \quad && \text{on $\partial \Omega_i$}, \\
			\lambda_i &= \frac{\partial u_i}{\partial n_i} \quad && \text{on $\partial \Omega_i$},\\
			u_0 &= 0 \quad && \text{on $\partial \Omega$},\\
		\end{aligned}
	\end{equation}
see \cite{baiocchi1992stabilization}. Here $n_i$ is the exterior unit vector of $\Omega_i$.
\end{remark}

\subsection{Finite element approximation with hybrid Nitsche}

Let $\Th_h$ be a shape regular finite element triangulation/tetrahedralization of $\Omega$ with maximum diameter $h>0$, and $\Th_{h,i}$ the submesh of $\Th_h$ corresponding to $\Omega_i$ with diameter $h_i$. We assume that there exist $c,C>0$ such that $ch\leq h_i\leq Ch$ for all $i=1,\dots,n.$ Further, let $V_i = \Set{w\in H^1(\Omega_i)\given w|_{\partial\Omega}=0}$ and consider finite-dimensional subspaces $V_{h,i}\subset V_i$, $V_{h,0}\subset V_0$, and $\varLambda_{h,i} \subset \varLambda_i$. It is possible to find a finite-dimensional approximate solution 
\begin{align*}
	u_h\in W_h & := \Pi_{i=1}^n V_{h,i} \times V_{h,0} \subset W, \\ 
	\lambda_h \in \varLambda_h & := \Pi_{i=1}^n \varLambda_{h,i} \subset \varLambda,
\end{align*}
for \eqref{eq:energyvariational}. In practice, we use spaces $V_{h,i} = \Set{w_h \in V_i \given w_h|_T\in P^p(T)\: \forall T\in \Th_{h,i}}$, where $p$ is the degree of the polynomial finite element basis.

We approximate problem \eqref{eq:energyvariational} by a stabilized FEM where by adding a stabilization term one avoids the Babuska-Brezzi condition and specially constructed finite element spaces, cf. \cite{barbosa1991finite,stenberg1995some}, or in the domain decomposition context \cite{becker2003finite,gustafsson2019error}. After stabilization, the
variational problem becomes: find $(u_h, \lambda_h)\in W_h\times \varLambda_h$ such that for all $(v_h, \mu_h)\in W_h \times \varLambda_h$
\begin{equation*}
	\begin{aligned}
		\label{eq:lagrange}
		\sum_{i=1}^n \bigg((\nabla u_{h,i}, \nabla v_{h,i})_{\Omega_i} - (\lambda_{h,i}, v_{h,i} - v_{h,0})_{\partial\Omega_i} - (\mu_{h,i}, u_{h,i} &- u_{h,0})_{\partial\Omega_i} \\
		- \alpha h_i \left(\lambda_{h,i} - \frac{\partial u_{h,i}}{\partial n_i}, \mu_{h,i} - \frac{\partial v_{h,i}}{\partial n_i}\right)_{\partial\Omega_i}\bigg) & = \sum_{i=1}^n (f, v_{h,i})_{\Omega_i},
	\end{aligned}
\end{equation*}
where $\alpha >0$ is a stabilization parameter.

Assume $\varLambda_{h,i}\subset V_{h,i}|_{\partial\Omega_i} = V_{h,0}|_{\partial\Omega_i}$. The stabilized Lagrange multiplier setting can be manipulated at each $\Omega_i$ into a hybrid Nitsche formulation by finding $\lambda_{h,i}=-\frac{1}{\alpha h_i}(u_{h,i}-u_{h,0}) + \frac{\partial u_{h,i}}{\partial n}$ and substituting $\mu_{h,i}=\frac{1}{\alpha h_i}(v_{h,i}-v_{h,0}) + \frac{\partial v_{h,i}}{\partial n}$, see \cite{gustafsson2019error,stenberg1995some}. 
We then define the desired bilinear form $\Bh_h : W_h\times W_h \to \R$ and linear form $\Fh: W\to \R$ as
\begin{equation}
	\begin{aligned}
		\label{eq:bilinf}
			\Bh_h(u_h, v_h) &= \sum_{i=1}^n \bigg((\nabla u_{h,i}, \nabla v_{h,i})_{\Omega_i} - \left(\frac{\partial u_{h,i}}{\partial n_i}, v_{h,i} - v_{h,0}\right)_{\partial\Omega_i} \\
							&\quad - \left(\frac{\partial v_{h,i}}{\partial n_i}, u_{h,i} - u_{h,0}\right)_{\partial\Omega_i} + \frac{1}{\alpha h_i} \left(u_{h,i} - u_{h,0}, v_{h,i} - v_{h,0}\right)_{\partial\Omega_i}\bigg),
	\end{aligned}
\end{equation}
\begin{equation*}
	\begin{aligned}
		\label{eq:linf}
		\Fh(v) = \sum_{i=1}^n (f, v_{i})_{\Omega_i}.
	\end{aligned}
\end{equation*}
The hybrid Nitsche variational problem is: find $u_h\in W_h$ such that 
\begin{equation}
	\begin{aligned}
		\label{eq:nitsche}
		\Bh_h(u_h, v_h) &= \Fh(v_h) \quad\forall v_h\in W_h.
	\end{aligned}
\end{equation}

\medskip

To recapitulate, our approach is the following. We approximate problem \eqref{eq:energy} by splitting $\Omega$ into subdomains and using a hybridized Nitsche finite element method. Through a local model order reduction, we are able to decrease the number of degrees-of-freedom of the approximation significantly and at the same time estimate  the local error with respect to a user-specified tolerance parameter. Our particular interest is in large-scale problems in complex geometries with at least 10 million degrees-of-freedom and our computational environment is restricted to a distributed setting, where compute nodes cannot communicate. On each independent compute node, we create a reduced basis for a subdomain by approximating the solution using a low-rank approximation of a lifting operator. The resulting lower-dimensional problem can be solved on a single large memory node. 

We next describe the local model order reduction scheme.


\section{Local model order reduction}
\label{sec:localmor}

This section follows closely our previous paper \cite{gustafsson2024}, albeit omitting some of the details. We first extend the partition $\Set{\Omega_i}_{i=1}^n$ given a mesh $\Th_h$. Then, we specify a low-rank approximation problem used in constructing the reduced spaces $\widetilde{V}_{h,i}\subset V_{h,i}$.


We begin by defining the required spaces. Recall that $\Th_{h,i}$ is a local mesh corresponding to the subdomain $\Omega_i$. Our domain decomposition method relies on overlapping subdomains. Hence, each local mesh is extended to a mesh $\Th_{h,i}^+$ by adding elements of $\Th_h$ that are within a fixed distance $r>0$ from $\Th_{h,i}$:
\begin{align*}
\Th_{h,i}^+ = \Set{T\in \Th_{h} \given \inf_{x\in T, y\in \Omega_i} \norm[\ell_2]{x-y}<r},
\end{align*}
where $\norm[\ell_2]{\cdot}$ refers to the Euclidean norm. Each extended local mesh defines an extended subdomain $\Omega_i^+$, where $\partial\Omega_i^+$ does not cut through any elements. The extended subdomains induce finite element spaces 
\begin{align*}
V_{h,i}^+ = \Set{w\in H^1(\Omega_i^+)\given w|_{\partial\Omega}=0, w|_T\in P^p(T)\:\forall T \in \Th_{h,i}^+}.
\end{align*}
In Section \ref{sec:numerics}, we demonstrate how increasing the extension parameter $r$ results in smaller reduced bases at the expense of larger local problems. 


In the extended subdomains, we define the local problems:
for $g_h\in V_{h,i}^+|_{\partial\Omega_i^+}$, find $w_{h,i}\in V_{h,i}^+$ as the finite element solution of 
\begin{equation}
\begin{aligned}
\label{eq:localproblem}
-\Delta w_i &= f \quad && \text{in $\Omega_i^+$},\\
w_i &= g_h \quad && \text{on $\partial \Omega_i^+$}.
\end{aligned}
\end{equation}
The solution $w_i$ can be written as a sum of two terms, where one accounts for the load $f$ with zero boundary condition and another for the boundary condition $g_h$ with zero load. This separation is essential in constructing the local reduced spaces $\widetilde{V}_{h,i}$, see Definition \ref{def:reducedspace} below.

The following restricted lifting operator provides a tool to find small bases using local finite element trace spaces, regardless of the unknown local boundary condition.
\begin{definition}[$\Zh_i$ operator]
Let $\Zh_i : V_{h,i}^+|_{\partial\Omega_i^+} \to V_{h,i}$ be such that for any $g_h\in V_{h,i}^+|_{\partial\Omega_i^+}$ 
the map $\Zh_i g_h = w_{h,i}^{g}|_{\Omega_i} \in V_{h,i}$ is the finite element solution of 
\begin{equation*}
\begin{aligned}
\label{eq:zeroload}
-\Delta w_i^g &= 0 \quad && \text{in $\Omega_i^+$},\\
w_i^g &= g_h \quad && \text{on $\partial \Omega_i^+$},
\end{aligned}
\end{equation*}
restricted to $\Omega_i$.
\end{definition}

Next, we define the norms for functions $v_h\in V_{h,i}$ and $g_h\in V_{h,i}^+|_{\partial\Omega_i^+}$
\begin{align}
\label{eq:hinorm}
\norm[h,i]{v_h} & = \left(\norm[0,\Omega_i]{\nabla v_h}^2 + \frac{1}{h_i}\norm[0,\partial\Omega_i]{v_h}^2\right)^{\frac{1}{2}} \\
\label{eq:tracenorm}
\norm[1/2,\partial\Omega_i^+]{g_h} & = \min_{\substack{v_h\in V_{h,i}^+\\ v_h|_{\partial\Omega_i^+}=g_h}}\norm[1, \Omega_i^+]{v_h},
\end{align}
where $\norm[0]{\cdot}$ is the $L^2$ norm and $\norm[1]{\cdot}$ the $H^1$ norm.
We can now define the low-rank approximation of $\Zh_i$ in the desired norm given the user-specified tolerance parameter.
\begin{definition}[Low-rank approximation of $\Zh_i$]
\label{def:lowrankapprox}
Fix $\epsilon>0$. Then, $\widetilde\Zh_i$ is defined as the lowest rank approximation of $\Zh_i$ that satisfies
\begin{align}
\label{eq:zopapprox}
\norm{\Zh_i-\widetilde\Zh_i} = \max_{g_h\in V_{h,i}^+|_{\partial\Omega_i^+}} \frac{\norm[h,i]{(\Zh_i - \widetilde \Zh_i)g_h}}{\norm[1/2, \partial\Omega_i^+]{g_h}}\leq \epsilon.
\end{align}
\end{definition}

The low-rank approximation $\widetilde\Zh_i$ implicitly defines a subspace $\widetilde V_{h,i}^+|_{\partial\Omega_i^+}\subset V_{h,i}^+|_{\partial\Omega_i^+}$, and we define the parameters
\begin{align*}
\dim(V_{h,i})&=m_i, \\ 
\dim(V_{h,i}^+|_{\partial\Omega_i^+})&=K_i, \\
\dim(\widetilde V_{h,i}^+|_{\partial\Omega_i^+})&=k_i,
\end{align*}
for which $k_i\ll K_i \ll m_i$. Notice that $\mathrm{rank}(\Zh_i)=K_i$ and $\mathrm{rank}(\widetilde\Zh_i) = k_i$. 
This brings us to the key definition, the reduced space $\widetilde V_{h,i}$.

\begin{definition}[Reduced space $\widetilde V_{h,i}$]
\label{def:reducedspace}
Let $f\in L^2(\Omega_i^+)$ and $w_{h,i}^f|_{\Omega_i}\in V_{h,i}$ be the finite element solution of 
\begin{equation*}
\begin{aligned}
-\Delta w_i^f &= f \quad && \text{in $\Omega_i^+$},\\
w_i^f &= 0 \quad && \text{on $\partial \Omega_i^+$},
\end{aligned}
\end{equation*}
restricted to $\Omega_i$. The reduced space $\widetilde V_{h,i}$ is defined as $$\widetilde V_{h,i} = \mathrm{span}(w_{h,i}^f|_{\Omega_i}) \oplus \mathrm{range}(\widetilde \Zh_i).$$
\end{definition}
The basis function $w_{h,i}^f$ can be solved trivially given $V_{h,i}^+$. The operator $\Zh_i$ and its low-rank approximation can be similarly constructed given the local extended finite element space. The reduced space $\widetilde V_{h,i} \subset V_{h,i}$ is problem dependent; while the generic finite element space $V_{h,i}$ can readily approximate any load from $L^2(\Omega_i)$, the reduced space $\widetilde V_{h,i}$ is specifically crafted to provide an approximate solution given load $f$. This decreases the basis by $m_i-K_i -1$ functions. The dimensionality is further reduced by the low-rank approximation of $\Zh_i$. 
Effectively, the other $K_i-k_i$ dimensions are treated as noise and zeroed out. Hence, the dimensionality reduction on a subproblem is $$\dim(V_{h,i}) = m_i \gg m_i - (m_i - K_i -1) - (K_i -k_i) = k_i + 1 = \widetilde{k_i} = \dim(\widetilde V_{h,i}).$$ Figure \ref{fig:zspectrum} showcases how the spectrum of $\Zh_i$ exhibits almost exponential decay so that the low-rank approximation and $\widetilde{k_i}$ are truly small. 

The finite element solutions from $V_{h,i}$ and $\widetilde V_{h,i}$ satisfy the following error bound. 
\begin{lemma}[Local error]
\label{lemma:localerror}
Fix $\epsilon >0$ and $g_h\in V_{h,i}^+|_{\partial\Omega_i^+}$. Let $w_{h,i}|_{\Omega_i}\in V_{h,i}$ and $\widetilde w_{h,i}|_{\Omega_i}\in \widetilde V_{h,i}$ be restrictions of finite element solutions to \eqref{eq:localproblem}. 
Further, let $w_{h,i}\in V_{h,i}^+$ such that $\norm[1/2, \partial\Omega_i^+]{g_h} = \norm[1,\partial\Omega_i^+]{w_{h,i}}$. Then 
\begin{align*}
\norm[h,i]{w_{h,i}|_{\Omega_i} - \widetilde w_{h,i}|_{\Omega_i}} \leq \epsilon\norm[1, \Omega_i^+]{w_{h,i}}.
\end{align*}
\end{lemma}
\begin{proof}
Per Definitions \ref{def:lowrankapprox} and \ref{def:reducedspace}, 
\begin{align*}
\norm[h,i]{w_{h,i}|_{\Omega_i} - \widetilde w_{h,i}|_{\Omega_i}} & =  \norm[h,i]{(w_{h,i}^f + w_{h,i}^g)|_{\Omega_i} - (w_{h,i}^f + \widetilde w_{h,i}^g)|_{\Omega_i}} \\
														 & = \norm[h,i]{w_{h,i}^g|_{\Omega_i} - \widetilde w_{h,i}^g|_{\Omega_i}} \\
														 & = \norm[h,i]{(\Zh_i - \widetilde \Zh_i)g_h} \\
														 & \leq \norm{\Zh_i - \widetilde \Zh_i}\norm[1/2, \partial\Omega_i^+]{g_h} \\
														 & \leq \epsilon\norm[1, \Omega_i^+]{w_{h,i}}.
\end{align*}
\end{proof}
Lemma \ref{lemma:localerror} presents how we can control the local error by choosing a suitable $\epsilon >0$. The result plays an important role in our final error estimate.

\begin{remark}
As we have defined $\Zh_i$ with respect to finite-dimensional spaces, there exist matrix representations of the above. In fact, given any two positive definite matrices $\bm M\in\R^{m\times m}, \bm N\in \R^{n\times n}$, we can define a norm for a general matrix $\bm A\in\R^{m\times n}$ through $$\norm[MN]{\bm A} = \norm[2]{\bm M^{\frac{1}{2}}\bm A\bm N^{-\frac{1}{2}}} = \max_{\substack{\bm x\in\R^n, \\ \bm x \neq \bm 0}}\frac{\norm[\ell_2]{\bm M^{\frac{1}{2}}\bm A\bm N^{-\frac{1}{2}}\bm x}}{\norm[\ell_2]{\bm x}},$$ where we use the spectral norm $\norm[2]{\cdot}$ induced by the Euclidean vector norm.

Let $\bm Z_i$ and $\bm{\widetilde Z}_i$ be the matrix representations of $\Zh_i$ and $\widetilde{\Zh}_i$, respectively. Then, the low-rank matrix approximation problem is: given $\epsilon > 0$ and $\bm{Z}_i$
, find $\bm{\widetilde Z}_i$ such that  
\begin{align}
\label{eq:zmatapprox}
\norm[MN]{\bm Z_i-\bm{\widetilde Z}_i}<\epsilon,
\end{align}
where $\bm M$ incorporates the norm $\norm[h,i]{\cdot}$ and $\bm N$ the norm $\norm[1/2, \Omega_i^+]{\cdot}$. By the classical Eckart-Young-Mirsky Theorem, the best rank $k$ approximation to $\bm Z_i$ is obtained by truncated SVD of $\bm M^{\frac{1}{2}}\bm Z_i \bm N^{-\frac{1}{2}}$. Then the error $\norm[MN]{\bm Z_i-\bm{\widetilde Z}_i}$ is equal to the $k+1$th largest singular value of $\bm M^{\frac{1}{2}}\bm Z_i \bm N^{-\frac{1}{2}}$. The difference here with respect to our previous paper \cite{gustafsson2024} is the change made in $\bm M$ due to the mesh-dependent norm $\norm[h,i]{\cdot}$. The technicalities how to compute the low-rank approximation \eqref{eq:zmatapprox} and build the reduced basis efficiently -- important details for large-scale computing -- are described in \cite{gustafsson2024}.
\end{remark}
\begin{remark}
    Another modification to our previous work is that the trace variable $u_{h,0}$ enforces the continuity instead of the stitching operator \cite[Definition 3.4]{gustafsson2024} that requires an additional layer of overlapping elements. This is a major simplification for local indexing and creating the global reduced basis in Section \ref{sec:implementation}.
\end{remark}

\section{Error analysis}
\label{sec:error}

Section \ref{sec:localmor} concentrated on local problems to construct local reduced spaces $\widetilde V_{h,i}, i=1,\dots,n$ given a user-specified tolerance, described in Definition \ref{def:reducedspace}. With these local reduced spaces, we can approximate the Nitsche solution with a solution from the global reduced space $\widetilde W_h := \Pi_{i=1}^n \widetilde V_{h,i} \times V_{h,0}\subset W_h$. Given $\epsilon >0$, the reduced Nitsche variational problem is: find $\tilde u_h\in \widetilde W_h$ such that
\begin{equation}
\begin{aligned}
\label{eq:rednitsche}
\Bh_h(\tilde u_h, \tilde v_h) &= \Fh(\tilde v_h)\quad \forall \tilde v_h\in\widetilde W_h .
\end{aligned}
\end{equation}

We now derive our error estimate for the reduced Nitsche solution in a mesh dependent norm. Several supporting results are developed after which Theorem \ref{thm:error} presents an error bound. The reduced solution $\tilde u_h$ converges to the continuous solution $u$ polynomially with respect to mesh diameter $h$ and linearly with respect to tolerance $\epsilon$.

The error estimate is derived for the $h$ norm
\begin{align}
\label{eq:hnorm}
\norm[h]{v_h} &= \left(\sum_{i=1}^n \norm[0, \Omega_i]{\nabla v_{h,i}}^2 + \frac{1}{h_i}\norm[0, \partial\Omega_i]{v_{h,i}-v_{h,0}}^2\right)^{\frac{1}{2}}.
\end{align}
\begin{remark}
Norm \eqref{eq:hnorm} includes terms estimating the difference between the trace of $u_{h,i}$ and $u_{h,0}$ on $\partial\Omega_i$. These terms allow us to show coercivity of the bilinear form \eqref{eq:bilinf} with respect to the mesh-dependent norm \eqref{eq:hnorm}.

Note that the previously defined mesh-dependent norm \eqref{eq:hinorm} was used only locally for defining the $\Zh_i$ operator and estimating the error of the low-rank approximation in Lemma \ref{lemma:localerror}.
\end{remark}

First, we present an auxiliary result that can easily be proven using a standard scaling argument. 
There exists a constant $C_I>0$, independent of $h_i$ and $i$, such that for all $i=1,\dots,n,$ it holds
\begin{align}
\label{eq:traceineq}
h_i\normbig[0, \partial\Omega_i]{\displaystyle\frac{\partial v_{h,i}}{\partial n_i}}^2 \leq C_I \norm[0,\Omega_i]{\nabla v_{h,i}}^2 \quad \forall v_{h,i}\in V_{h,i}.
\end{align}

We continue with a lemma on stability. In what follows, we denote by $C$ a generic positive constant, independent of $h$, whose value may change from estimate to estimate.  
\begin{lemma}[Coercivity of $\Bh_h$]
\label{lemma:coercivity}
Assume $0 < \alpha < C_I^{-1}$. Then 
\begin{align*}
\Bh_h(v_h, v_h) \geq C\norm[h]{v_h}^2 \quad \forall v_h\in W_h.
\end{align*}
\end{lemma}
\begin{proof}
Using Cauchy-Schwarz, \eqref{eq:traceineq} and Young's inequality,
\begin{align*}
& \Bh_h(v_h, v_h) \\
& = \sum_{i=1}^n \bigg((\nabla v_{h,i}, \nabla v_{h,i})_{\Omega_i} - 2\left(\frac{\partial v_{h,i}}{\partial n}, v_{h,i}-v_{h,0}\right)_{\partial\Omega_i} + \frac{1}{\alpha h_i} \left(v_{h,i}-v_{h,0}, v_{h,i}-v_{h,0}\right)_{\partial\Omega_i}\bigg) \\
& \geq \sum_{i=1}^n \bigg(\norm[0,\Omega_i]{\nabla v_{h,i}}^2 - 2\sqrt{h_i}\normbig[0, \partial\Omega_i]{\frac{\partial v_{h,i}}{\partial n}} \frac{1}{\sqrt{h_i}}\norm[0, \partial\Omega_i]{v_{h,i}-v_{h,0}} + \frac{1}{\alpha h_i} \norm[0, \partial\Omega_i]{v_{h,i}-v_{h,0}}^2\bigg) \\
& \geq \sum_{i=1}^n \bigg(\norm[0,\Omega_i]{\nabla v_{h,i}}^2 - \delta h_i\normbig[0, \partial\Omega_i]{\frac{\partial v_{h,i}}{\partial n}}^2 - \frac{1}{\delta h_i}\norm[0, \partial\Omega_i]{v_{h,i}-v_{h,0}}^2 + \frac{1}{\alpha h_i} \norm[0, \partial\Omega_i]{v_{h,i}-v_{h,0}}^2\bigg) \\
& \geq \sum_{i=1}^n \bigg(\left(1-\delta C_I\right)\norm[0,\Omega_i]{\nabla v_{h,i}}^2 + \left(\frac{1}{\alpha}-\frac{1}{\delta}\right) \frac{1}{h_i}\norm[0, \partial\Omega_i]{v_{h,i}-v_{h,0}}^2\bigg) \\
& \geq C \norm[h]{v_h}^2,
\end{align*}
when $\delta \in (\alpha, C_I^{-1})$. 
\end{proof}

Building on this result, C\'ea's lemma is proved for the $h$ norm and spaces $W_h$ and $\widetilde W_h$.
\begin{lemma}[C\'ea's lemma]
\label{lemma:cea}
Fix $\epsilon >0.$ Let $u_h\in W_h$ and $\tilde u_h\in \widetilde W_h$ be the unique solution to \eqref{eq:rednitsche}. Then 
\begin{align*}
\norm[h]{u_h - \tilde u_h} \leq C\norm[h]{u_h - \tilde v_h} \quad \forall \tilde v_h\in \widetilde W_h.
\end{align*}
\end{lemma}
\begin{proof}
Using Lemma \ref{lemma:coercivity}, Galerkin orthogonality and continuity of the bilinear form,
\begin{align*}
\norm[h]{u_h - \tilde u_h}^2 & \leq C\Bh_h(u_h - \tilde u_h, u_h - \tilde u_h) \\
					 & \leq C\Bh_h(u_h - \tilde u_h, u_h - \tilde v_h) \\
					 & \leq C\norm[h]{u_h - \tilde u_h}\norm[h]{u_h - \tilde v_h}.
\end{align*}
\end{proof}

The convergence of the (hybrid) Nitsche solution to the continuous solution is stated next. Here and in the following we denote by $u^C$ the solution to \eqref{eq:energy} and by $u\in W$ the corresponding function in $W$.
\begin{lemma}[Nitsche error in $h$ norm]
\label{lemma:nitschehnormerror}
Let $p\geq 1$ be the polynomial degree and \\$u^C\in H_0^1(\Omega)\cap H^{p+1}(\Omega)$ be the unique solution to \eqref{eq:energy}, $u=(u^C|_{\Omega_1},\ldots,u^C|_{\Omega_n},u^C|_\Gamma) \in W$ and $u_h\in W_h$ be the unique solution to \eqref{eq:nitsche}. Then 
$$\norm[h]{u-u_h} \leq C h^p\norm[p+1]{u^C}.$$
\end{lemma}
\begin{proof}
See \cite[Theorem 3]{oikawa2010}.
\end{proof}

Let $V_h^C = \Set{w\in H_0^1(\Omega)\given w|_T\in P^p(T)\:\forall T \in \Th_h} \subset H^1_0(\Omega)$. We then define the conforming finite element solution through the problem: find $u_h^C\in V_h^C$ such that
\begin{align}
\label{eq:conffem}
(\nabla u_h^C, \nabla v_h^C)_\Omega = (f,v_h^C)_\Omega \quad\forall v_h^C \in V_h^C.
\end{align}
The conforming solution $u_h^C$ approximates $u^C$ with an error bound $$\norm[1,\Omega]{u^C-u_h^C}\leq C h^p\norm[p+1]{u^C}$$ \cite{brennermathematical2008}
and is used to prove a reduction error estimate. 
In addition, we make a modest assumption on the overlapping subdomains. 
\begin{assumption}[Partition overlap]
\label{ass:overlap}
Let the triangulation/tetrahedralization $\Th_h$ be partitioned to $n$ overlapping local meshes $\Th_{h,i}^+$. Assume that any element $T\in\Th_h$ is included in at most $m \ll n$ local meshes. 
\end{assumption}
Assumption \ref{ass:overlap} is satisfied trivially when the mesh is large and the extension parameter $r$ is kept within reasonable limits, such as under half of the subdomain diameter. 

Notice, in particular, that as all finite element spaces are defined via the same mesh $\Th_h$ and have the same polynomial degree basis, $V_h^C|_{\partial\Omega_i^+}$ and $V_{h,i}^+|_{\partial\Omega_i^+}$ are, in fact, the same spaces. Consequently, Assumption \ref{ass:overlap} ensures that $$\sum_{i=1}^n\norm[1,\Omega_i^+]{u_h^C} \leq m \norm[1,\Omega]{u_h^C},$$ and the restriction of the conforming solution to $\Omega_i^+$ coincides with $w_{h,i}$ in Lemma \ref{lemma:localerror}. Further, the conforming solution can be written as $(u_h^C|_{\Omega_1},\dots, u_h^C|_{\Omega_n}, u_h^C|_\Gamma)$.


We proceed to prove the finite element reduction error in the mesh-dependent norm by choosing a test function such that the reduction trace error vanishes.
\begin{lemma}[Reduction error in $h$ norm]
\label{lemma:reductionhnormerror}
Let $u_h^C \in V_h^C$ be the conforming finite element solution to \eqref{eq:conffem}. Suppose the mesh $\Th_h$ satisfies Assumption \ref{ass:overlap} and that $h\in (0,1]$. Let $\epsilon >0$ and denote 
\begin{align}
\label{eq:redtestfunction}
\tilde v_h = (\tilde v_{h,1}, \dots, \tilde v_{h,n}, u_h^C|_\Gamma) = (\widetilde \Zh_1(u_h^C|_{\partial\Omega_1^+}), \dots, \widetilde \Zh_n(u_h^C|_{\partial\Omega_n^+}), u_h^C|_\Gamma)\in \widetilde W_h.
\end{align} 
Then 
$$\norm[h]{u_h^C - \tilde v_h} \leq C\epsilon m \norm[p+1]{u^C}.$$
\end{lemma}
\begin{proof}
Using Lemma \ref{lemma:localerror} and Assumption \ref{ass:overlap},
\begin{align*}
\norm[h]{u_h^C - \tilde v_h}^2 & = \sum_{i=1}^n \norm[0, \Omega_i]{\nabla(u_{h,i}^C - \tilde v_{h,i})}^2 + \frac{1}{h_i}\norm[0, \partial\Omega_i]{u_{h,i}^C - \tilde v_{h,i}}^2 \\
					   & = \sum_{i=1}^n \norm[h,i]{u_{h,i}^C - \tilde v_{h,i}}^2 \\
						 & \leq \sum_{i=1}^n \epsilon^2 \norm[1, \Omega_i^+]{u_h^C}^2 \\
						 & \leq \epsilon^2 m^2 \norm[1, \Omega]{u_h^C}^2.
\end{align*}
The trace terms vanish as the test function equals the conforming finite element solution at $\partial\Omega_i, i=1,\dots,n$. Then, the norm of the conforming finite element solution can be bounded by
\begin{align*}
   \norm[1]{u_h^C} & = \norm[1]{u_h^C - u^C + u^C} \\
						  & \leq \norm[1]{u_h^C - u^C} + \norm[1]{u^C} \\
						  & \leq Ch^p\norm[p+1]{u^C} + \norm[p+1]{u^C} \\
						  & = C(h^p+1)\norm[p+1]{u^C}.
\end{align*}
Inserting this above yields 
\begin{align*}
   \norm[h]{u_h^C - \tilde v_h}^2 & \leq C\epsilon^2 m^2 (h^p+1)^2\norm[p+1]{u^C}^2 \leq C\epsilon^2 m^2\norm[p+1]{u^C}^2,
\end{align*}
as by assumption $h \leq 1.$ Taking a square root finishes the proof.
\end{proof}

Finally, we present our error estimate.
\begin{theorem}[Error estimate]
\label{thm:error}
Let $p\geq1$ be the polynomial degree and $\Th_h$ satisfy Assumption \ref{ass:overlap} with $h\in (0,1]$. Denote $u^C\in H^1_0(\Omega)\cap H^{p+1}(\Omega)$ the unique solution to \eqref{eq:energy},
$u=(u^C|_{\Omega_1},\ldots,u^C|_{\Omega_n},u^C|_\Gamma) \in W$,
and $\tilde u_h\in \widetilde W_h$ the reduced Nitsche solution to \eqref{eq:rednitsche} with tolerance $\epsilon > 0$. Then 
\begin{align}
   \label{eq:errorestimate}
   \norm[h]{u-\tilde u_h} & \leq C(h^p + \epsilon m)\norm[p+1]{u^C}.
\end{align}
\end{theorem}
\begin{proof}
We utilize repeated summing of zero and triangle inequality, and Lemma \ref{lemma:cea} to get
\begin{align*} 
   \norm[h]{u-\tilde u_h} & \leq \norm[h]{u-u_h} + \norm[h]{u_h - \tilde u_h} \\
						  & \leq \norm[h]{u-u_h} + C\norm[h]{u_h - \tilde v_h} \\
						  & \leq \norm[h]{u-u_h} + C\norm[h]{u_h - u_h^C} + C\norm[h]{u_h^C - \tilde v_h} \\
						  & \leq \norm[h]{u-u_h} + C\norm[h]{u_h - u} + C\norm[h]{u - u_h^C} + C\norm[h]{u_h^C - \tilde v_h} \\
						  & \leq C\norm[h]{u-u_h} + C\norm[0]{\nabla (u^C - u_h^C)} + C\norm[h]{u_h^C - \tilde v},
\end{align*}
where $u_h\in W_h$ is the Nitsche solution, $u_h^C\in V_h^C$ is the conforming finite element solution with $\norm[h]{u-u_h^C} = \norm[0]{\nabla (u^C-u_h^C)}$ as $u^C$ and $u_h^C$ are continuous on $\Gamma$, and $\tilde v_h\in \widetilde W_h$ is the function \eqref{eq:redtestfunction}. 

Then, using Lemmas \ref{lemma:nitschehnormerror} and \ref{lemma:reductionhnormerror} 
\begin{align*}
   \norm[h]{u-\tilde u_h} & \leq C\norm[h]{u-u_h} + C\norm[0]{\nabla (u^C - u_h^C)} + C\norm[h]{u_h^C - \tilde v} \\
						  & \leq Ch^p \norm[p+1]{u^C} + Ch^p\norm[p+1]{u^C} + C\epsilon m\norm[p+1]{u^C} \\
& \leq C(h^p + \epsilon m)\norm[p+1]{u^C}.
\end{align*}
\end{proof}

Theorem \ref{thm:error} provides an upper bound to the reduced solution that is dependent on the degree $p$ of the polynomial basis and the local dimension reduction tolerance $\epsilon$. Given a small enough tolerance $\epsilon$, the reduced approximations should practically coincide with the conventional FEM solution. Our numerical tests in Section \ref{sec:numerics} support this conclusion. 

\begin{remark}
The tolerance $\epsilon$ in \eqref{eq:zopapprox} provides an upper bound for the low-rank approximation, but leads to a very crude estimate in almost all practical cases. In the proof of Lemma~\ref{lemma:localerror}, when we bound the mesh-dependent norm in terms of the operator norm, we implicitly cover all possible boundary conditions on extended domains, including extremely pathological cases. Thus, for typical smoother loads the local reduced approximations can be several magnitudes more accurate.
\end{remark}

\section{Matrix implementation}
\label{sec:implementation}

Recall from Section \ref{sec:localmor} that $m_i = \dim(V_{h,i})$, and let $K = \dim(V_{h,0})$. The variational problem \eqref{eq:nitsche} can be written in matrix form as 
\begin{align}
	\label{eq:nitschemat}
		\underbrace{\mat{\bm{A}, \bm{B}; \bm{B}^T, \bm{C}}}_{\bm K} \vec{\bm{\beta}, \bm{\beta}_0}= \mat{\bm{A}_{1}, \bm{0}, \cdots, \bm{B}_{1}; \bm{0}, \bm{A}_{2}, \cdots, \bm{B}_2; \vdots, \vdots, \ddots, \vdots; \bm{B}_1^T, \bm{B}_2^T, \cdots, \bm{C}} \vec{\bm{\beta}_1, \bm{\beta}_2, \vdots, \bm{\beta}_0}& = \vec{\bm{f}_1, \bm{f}_2, \vdots, \bm{0}} = \vec{\bm{f}, \bm{0}},
\end{align}
where the matrices are of dimensions $\bm A_i\in \R^{m_i\times m_i}, \bm B_i\in \R^{m_i\times K}, i=1,\dots,n,$ $\bm C\in \R^{K\times K}$, and the vectors $\bm \beta_i, \bm f_i\in\R^{m_i}, i=1,\dots,n,$ and $\bm\beta_0\in\R^K$.
The elements of the different matrices are detailed in Appendix \ref{apx:matrixelements}.

The system \eqref{eq:nitschemat} can be solved in two steps:
\begin{equation}
	\begin{aligned}
\label{eq:nitschesystem}
		(\bm C - \bm B^T \bm A^{-1}\bm B)\bm \beta_0 & = -\bm B^T\bm A^{-1}\bm f, \\
		\bm \beta & = \bm A^{-1}(\bm f - \bm B\bm \beta_0).
	\end{aligned}
\end{equation}
Observe that the coefficient matrix $\bm K$ in \eqref{eq:nitschemat} is a matrix representation of $\Bh$ that is symmetric and coercive in the mesh-dependent norm and hence symmetric positive definite. Thus, the Schur complement matrix $\bm S = \bm C - \bm B^T\bm A^{-1}\bm B$ is also symmetric and positive definite. 
In particular, $\bm y^T \bm S \bm y = \bm x^T \bm K \bm x  > 0$ for any $\bm y\in\R^K$ with a corresponding $\bm x = \bm x(\bm y)$ as 
$$\bm y^T\bm S \bm y = 
\underbrace{\bm y^T \mat{\bm 0 \bm I}\mat{\bm I \bm 0; -\bm B^T\bm A^{-1} \bm I}}_{\bm x^T} \underbrace{\mat{\bm A \bm B; \bm B^T \bm C}}_{\bm K}\underbrace{\mat{\bm I \bm 0; -\bm B^T\bm A^{-1} \bm I}^T \mat{\bm 0;\bm I}\bm y}_{\bm x}.$$
Hence, we can use, e.g., the conjugate gradient method to solve for $\bm\beta_0$. Then, the value of $\bm\beta_0$ can be inserted to find $\bm\beta$.

The system \eqref{eq:nitschemat} can also be reduced per our local model order reduction scheme as described in Section \ref{sec:localmor}. For each subdomain, we can find a reduced basis such that $\bm Q_i\in \R^{m_i\times \widetilde{k_i}}, \widetilde{k_i}\ll m_i,$ and $\bm Q_i^T \bm A_{i}\bm Q_i = \bm\Lambda_i\in\R^{\widetilde{k_i}\times \widetilde{k_i}}$ is a diagonal matrix. Further, $\bm Q_i^T \bm B_i = \widetilde{\bm B_i}\in \R^{\widetilde{k_i}\times K}$ and $\bm Q_i\bm f_i = \widetilde{\bm f_i}\in\R^{\widetilde{k_i}}.$ Defining $\tilde k = \sum_{i=1}^n \tilde k_i$, the full reduced matrices are then $\widetilde{\bm B}\in\R^{K\times \tilde k}, \bm \Lambda\in\R^{\tilde k\times \tilde k}$ and $\tilde{\bm f}\in\R^{\tilde k}.$ We refer to our previous paper \cite[Section 6]{gustafsson2024} for more details on how to compute the reduced bases efficiently with randomized numerical linear algebra \cite{halkofinding2011,martinssonrandomized2020}.

\begin{figure}
	\begin{center}
		\includegraphics[width=\textwidth]{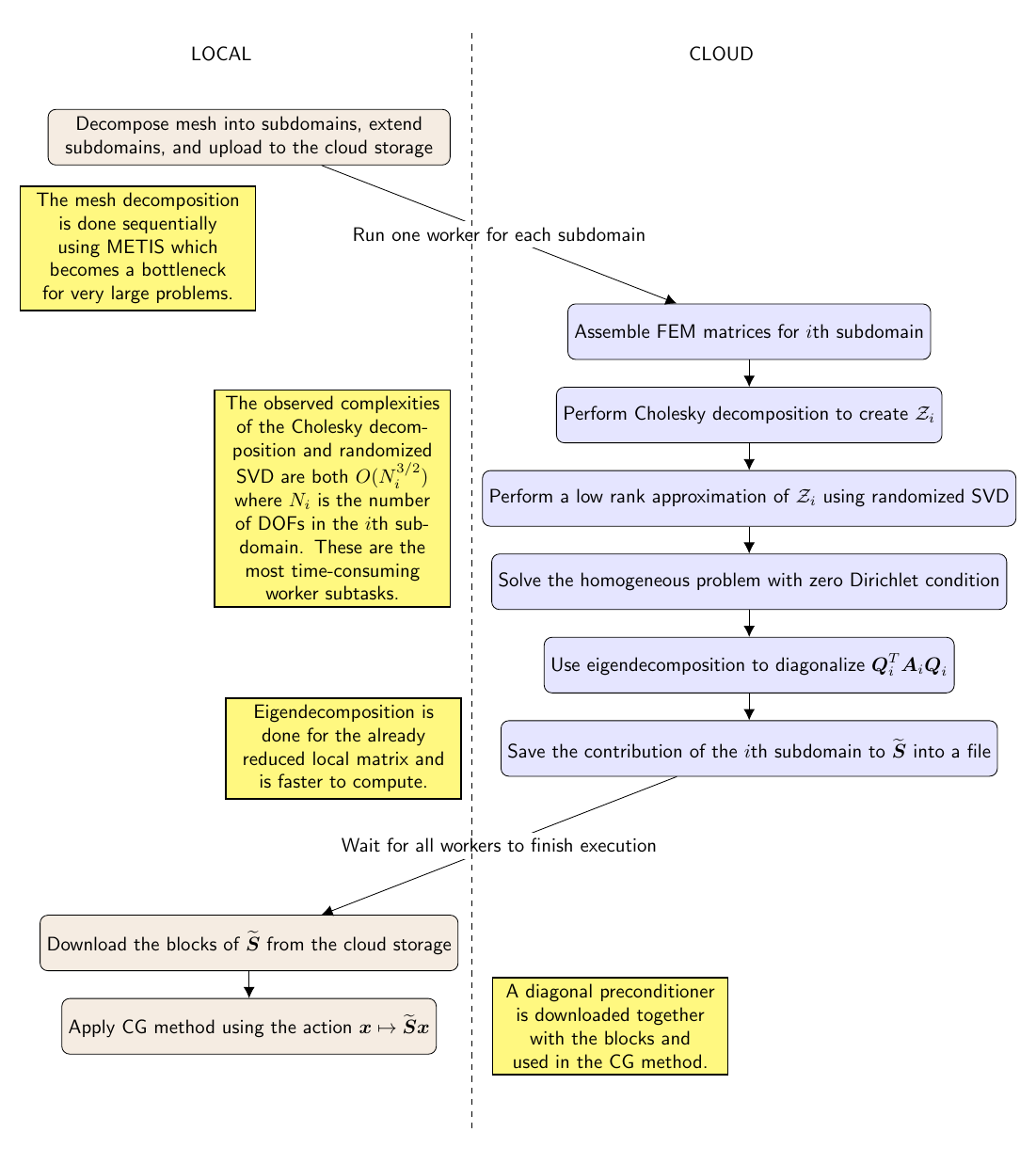}
	\end{center}
	\caption{The computational implementation on a high level and its separation between a local main node and distributed cloud worker nodes. For more details, see \cite[Section 6]{gustafsson2024} and our code \cite{sourcepackage}. The most expensive worker subtask complexities were analyzed empirically in \cite[Figure 6]{gustafsson2024}.}
	\label{fig:implementation}
\end{figure}

The solution \eqref{eq:nitschesystem} is then approximated by 
\begin{equation}
	\begin{aligned}
		\label{eq:nitschesystemred}
		(\bm C-\widetilde{\bm B}^T\bm\Lambda^{-1}\widetilde{\bm B})\bm\beta_0 & = -\widetilde{\bm B}^T\bm\Lambda^{-1}\widetilde{\bm f}, \\
		\tilde{\bm\beta} & = \bm\Lambda^{-1}(\widetilde{\bm f} - \widetilde{\bm B}\bm\beta_0).
	\end{aligned}
\end{equation}
We refer to \eqref{eq:nitschesystem} as the Nitsche system and to \eqref{eq:nitschesystemred} as the reduced Nitsche system. Note that the local basis reduction can be done independently for each subdomain. The global reduced matrices are simply formed block-wise on the main node afterwards, contrary to the involved projection scheme in \cite{gustafsson2024}. Finally, $\widetilde{\bm S} = \bm C - \widetilde{\bm B}^T\bm\Lambda^{-1}\widetilde{\bm B}$ is not constructed explicitly when using the preconditioned conjugate gradient method. Instead, the matrices $\bm C, \widetilde{\bm B}$ and $\bm \Lambda^{-1}$ are kept in memory which reduces both floating point operations and memory requirements as fewer nonzero elements are retained compared to computing $\widetilde{\bm S}$. The high-level workflow is described in Figure \ref{fig:implementation}.

\begin{remark}
	The most important enhancement from \eqref{eq:nitschesystem} to \eqref{eq:nitschesystemred} is reducing $\bm A^{-1}$ to $\bm \Lambda^{-1}$. $\bm A$ is a very large block-diagonal matrix that could be inverted locally block-by-block, but then the iterative solution methods would again require a supercomputer environment due to memory requirements. $\bm\Lambda$ is a much smaller diagonal matrix fitting into main node memory and analogous to the coarse problem in iterative substructuring methods \cite{farhat2001feti,dohrmann2003preconditioner}.
\end{remark}

\begin{remark}
    Compared to our previous work \cite{gustafsson2024}, the hybrid Nitsche scheme has clear implementational advantages. 
    The formulation of the local problems  is significantly tidier and the global projection to the reduced system is straightforward. Our previous method overlapped the subdomain DOFs with an extra layer during the local computations and the partition-of-unity-based approach grew requires specialized assembly. The hybrid Nitsche method enforces continuity with the trace variable in a simple and identical way for any polynomial basis.  Moreover, in \cite{gustafsson2024} the projection to the global reduced system had to be restricted to the intersections of the subdomain interfaces and parallelized to scale the method due to the cost of explicit matrix multiplications. This remained expensive and became the limiting factor for scaling, while for the hybrid Nitsche method the projection corresponds just to concatenating block matrices on the main node as all projections are done locally.
\end{remark}

\section{Numerical experiments}
\label{sec:numerics}

In this section, we confirm numerically the error estimate \eqref{eq:errorestimate} on the unit cube. Our code utilizes the \texttt{scikit-fem} package \cite{gustafssonscikitfem2020} and can be found from \cite{sourcepackage}. The method exhibits the expected polynomial convergence rate until it plateaus to the reduction error, which is much smaller than the theoretical local error for relatively smooth loads. Then, a large-scale computation with $20$ million degrees-of-freedom validates the scaling of the method, and gives rough requirements for the computational nodes. Finally, a model problem of a curved pipe showcases how the method performs with engineering geometries.

\subsection{Convergence tests}

Consider the problem \eqref{eq:energy} with load
\begin{align}
	\label{eq:testload}
	f = 2\sqrt{900}((1-x)x(1-y)y + (1-x)x(1-z)z + (1-y)y(1-z)z)
\end{align}
on $\Omega = [0,1]^3$. The energy norm of the analytical solution $u^C$ is equal to 1, and using Galerkin orthogonality the error in the mesh-dependent norm \eqref{eq:hnorm} reduces to  
\begin{equation}
    \begin{aligned}
    \label{eq:energynormerror}
        \norm[h]{u-\tilde u_h} & \leq C\left(\Bh_h(u-\tilde u_h, u-\tilde u_h)\right)^{1/2} \\
            & = C\left(\Bh_h(u,u) -\Bh_h(\tilde u_h, \tilde u_h)\right)^{1/2} \\
            & = C\left(\sum_{i=1}^n\norm[0,\Omega_i]{\nabla u}^2 - \norm[0,\Omega_i]{\nabla\tilde u_{h,i}}^2 + \frac{1}{h_i}\norm[0,\partial\Omega_i]{\tilde u_{h,0}-\tilde u_{h,i}}\right)^{1/2} \\
            & \approx \left(1 - \sum_{i=1}^n\norm[0,\Omega_i]{\nabla\tilde u_h}^2\right)^{1/2}.
    \end{aligned}
\end{equation} 
We have numerically observed that the error on the skeleton is comparable to noise when we use matching meshes and the same polynomial degree basis for all $u_{h,i}, i=1,\dots,n,$ and $u_{h,0}$. Thus, it is omitted from \eqref{eq:energynormerror}.

We approximate $u^C$ with $\tilde u_h$ by solving
\eqref{eq:rednitsche} via \eqref{eq:nitschesystemred}. Our implementation utilized the software package \texttt{scikit-fem} \cite{gustafssonscikitfem2020}.
Each subdomain was extended $r=4h$ using a $kd$-tree \cite{bentleymultidimensional1975}. 
We use $p=2$ and $\alpha=0.01$ to test the convergence with respect to mesh parameter $h$ and tolerance $\epsilon$. The number of subdomains was chosen such that each subdomain included roughly few thousand degrees-of-freedom and the extensions almost a magnitude more. The resulting reduced systems were solved using the conjugate gradient method with a simple diagonal preconditioner. The results are displayed in Figure \ref{fig:hepsilonplot} and Table \ref{tbl:convergence}.

\begin{table}[H]
\centering
	\caption{Figure \ref{fig:hepsilonplot} convergence test parameters and results on the unit cube $\Omega = [0,1]^3$ with linearly spaced discretizations. The columns $E_{\epsilon}$ showcase the errors \eqref{eq:energynormerror} for the three different tolerances $\Set{\expnumber{1}{-2}, \expnumber{1}{-3}, \expnumber{1}{-4}}$. The last columns present the number of conjugate gradient iterations for solving \eqref{eq:nitschesystemred} with a diagonal preconditioner and the condition number of the reduced Schur complement, respectively.}
\label{tbl:convergence}
\begin{tabular}{crccccccc}
\hline\noalign{\smallskip}
Case & $\dim(V)$ & $n$ & $h$ & $E_{\epsilon=\expnumber{1}{-2}}$ & $E_{\epsilon=\expnumber{1}{-3}}$ & $E_{\epsilon=\expnumber{1}{-4}}$ & Iter$_{CG}$ & $\kappa(\widetilde{\bm S})$\\
\noalign{\smallskip}\hline\noalign{\smallskip}
1 & \numprint{24389} & \numprint{10} & $\expnumber{1.2}{-1}$ & $\expnumber{7.7}{-3}$ & $\expnumber{7.7}{-3}$ &  $\expnumber{7.7}{-3}$ &  107 & $\expnumber{9.7}{+2}$ \\
2 & \numprint{91125} & \numprint{50}& $\expnumber{7.9}{-2}$ & $\expnumber{3.2}{-3}$ & $\expnumber{3.1}{-3}$  &  $\expnumber{3.1}{-3}$ & 194 & $\expnumber{3.0}{+3}$ \\
3 & \numprint{389017} & \numprint{200}& $\expnumber{4.8}{-2}$ & $\expnumber{1.4}{-3}$ & $\expnumber{1.2}{-3}$  &  $\expnumber{1.2}{-3}$ & 310 & $\expnumber{7.0}{+3}$ \\
4 & \numprint{1601613} & \numprint{700}& $\expnumber{3.0}{-2}$ & $\expnumber{9.3}{-3}$ & $\expnumber{4.5}{-4}$  &  $\expnumber{4.5}{-4}$ & 467 & $\expnumber{1.8}{+4}$ \\
5 & \numprint{7880599} & \numprint{2000}& $\expnumber{1.7}{-2}$ & $\expnumber{7.2}{-4}$ & $\expnumber{1.7}{-4}$  & $\expnumber{1.5}{-4}$ & 701 & - \\
\noalign{\smallskip}\hline
\end{tabular}
\end{table}

\begin{figure}
	\begin{center}
		\includegraphics[width=0.95\textwidth]{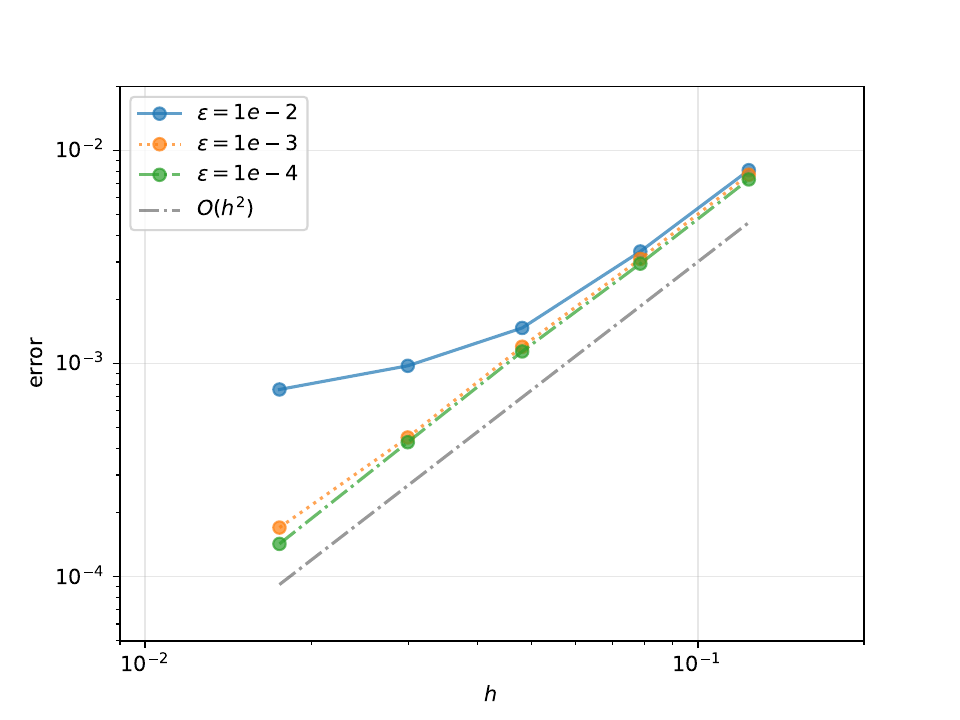}
	\end{center}
	\caption{Convergence of the method using second-degree polynomials on a log-log scale. On the $x$-axis is the mesh parameter $h$, and on the $y$-axis is the error \eqref{eq:energynormerror} of the approximation. The lines depict approximations $\tilde u_h$ with different reduction tolerance parameter $\epsilon$, and the gray line has the slope of the theoretical FEM convergence rate. The approximation error converges quadratically in $h$ as expected until the reduction error becomes the dominant factor for tolerance $\epsilon=\expnumber{1}{-2}$.}
	\label{fig:hepsilonplot}
\end{figure}


\begin{remark}
    Table \ref{tbl:convergence} shows the iterations using the conjugate gradient method with just a simple diagonal preconditioner.  The number of $CG$ iterations is large, hence the development and analysis of specialized preconditioners for the reduced Schur complement system is necessary and will be a topic of future work.
\end{remark}

The convergence follows the theoretical estimate for FEM with second-degree polynomial basis except for the larger tolerance $\epsilon=\expnumber{1}{-2}$ when the mesh parameter $h$ decreases enough. Then the reduction error term $\epsilon m\norm[p+1]{u^C}$ in Theorem \ref{thm:error} starts to dominate. The reduction error, i.e. the difference between the conforming finite element solution $u_h^C$ and the reduced Nitsche solution $\tilde u_h$, can be, similarly to \eqref{eq:energynormerror}, reduced to 
\begin{align}
	\label{eq:reductionerror}
	\norm[h]{u_h^C - \tilde u_h} \approx \left(\sum_{i=1}^n\left|\norm[0,\Omega_i]{\nabla u_h^C}^2 - \norm[0,\Omega_i]{\nabla \tilde u_h}^2\right|\right)^{1/2}.
\end{align} 
Figure \ref{fig:rederrorplot} exhibits the errors \eqref{eq:energynormerror} and the reduction errors \eqref{eq:reductionerror} for the approximations in Figure \ref{fig:hepsilonplot} and Table \ref{tbl:convergence}.

\vspace{-2em}
\begin{figure}[H]
	\begin{center}
		\includegraphics[width=0.95\textwidth]{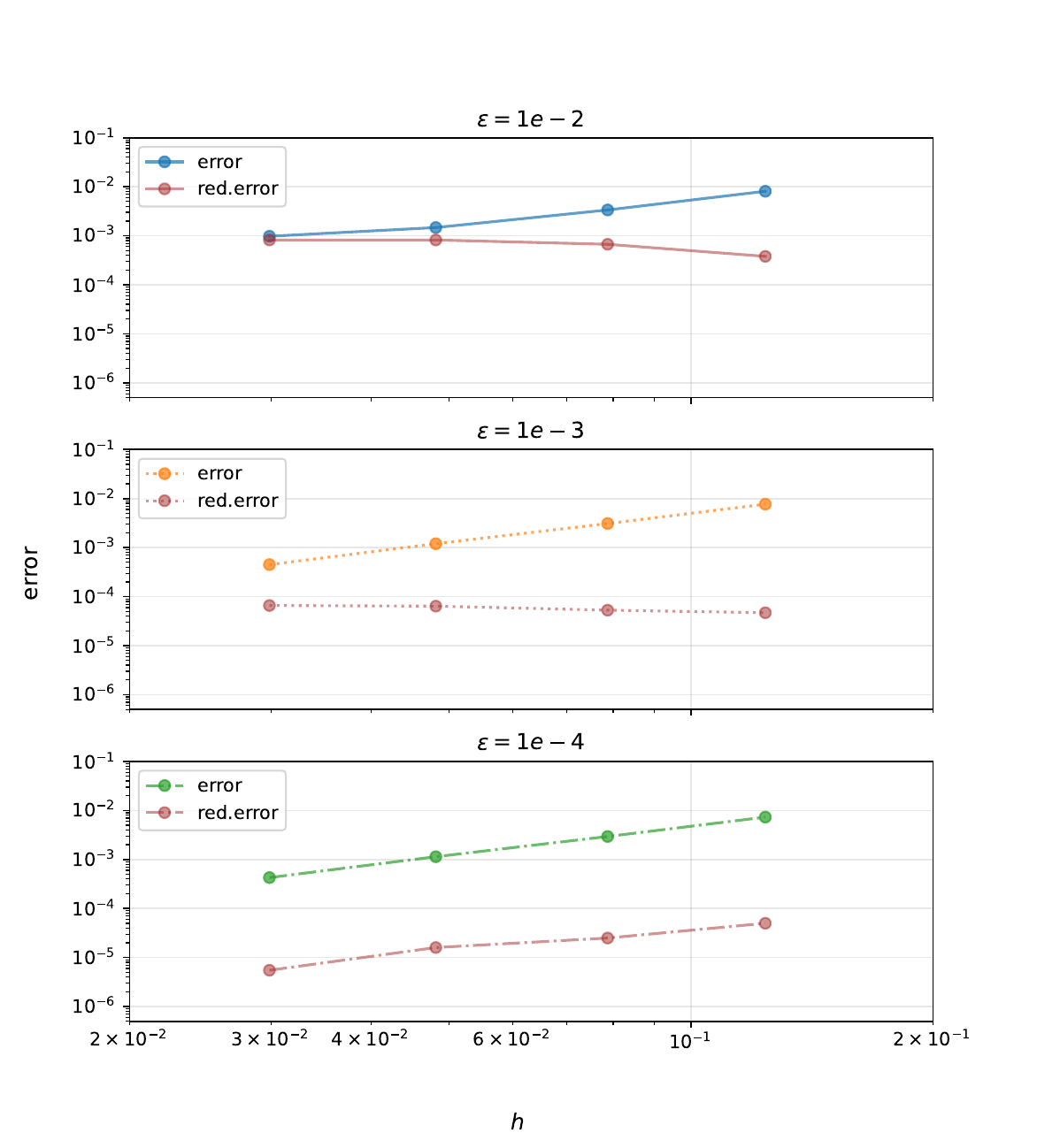}
	\end{center}
	\caption{Errors and reduction errors  for the approximations $\tilde u_h$ in Figure \ref{fig:hepsilonplot} with different tolerances $\epsilon$ on a log-log scale. Each subplot displays the error \eqref{eq:energynormerror} (varying color line) of the respective $\tilde u_h$ and the reduction error \eqref{eq:reductionerror} (dark red line). For the smaller tolerances the reduction error does not affect the approximation, but for $\epsilon=\expnumber{1}{-2}$ it is larger than the conventional FEM error for smaller $h$ and becomes the dominating factor.}\label{fig:rederrorplot}
\end{figure}

The reduction errors are relatively stable, but for $\epsilon=\expnumber{1}{-2}$ it is clear that already for $h\approx \expnumber{5}{-2}$ FEM is so accurate that the tolerance is too large and the basis reduction discards too much information. For load \eqref{eq:testload}, the reduction error seems to be between one and two magnitudes smaller than the tolerance $\epsilon$.

The tolerance $\epsilon$ determines the degree of dimension reduction, as it is the cutoff point to include only singular vectors of weighted $\bm Z_i$ that have singular values greater than $\epsilon$. Figure \ref{fig:zspectrum} presents the singular values of weighted $\bm Z_i$ given different extension parameters $r$ with otherwise the same load and parameterization. The subdomains were extended by a multiple of $h$, i.e. $r = ah, a\in\N$. The number of singular vectors in the reduced basis for different extensions $r$ and tolerance $\epsilon$ are depicted in Table \ref{tbl:localcutoffs}.

\vspace{-0.5em}
\begin{table}[h]
\centering
\caption{Number of singular vectors $k_i$ in the local reduced basis for different extensions $r$ and tolerance $\epsilon$. The original subdomain in question had $m_i=\numprint{3045}$ DOFs. The complete spectra are presented in Figure \ref{fig:zspectrum}.} 
\label{tbl:localcutoffs}
\begin{tabular}{lllcr}
\hline\noalign{\smallskip}
$\epsilon$ & $2h$ & $3h$ & $4h$ \\
\noalign{\smallskip}\hline\noalign{\smallskip}
$\expnumber{1}{-2}$ & 161 & 58 & 29 \\
$\expnumber{1}{-3}$ & 331 & 128 & 62 \\
$\expnumber{1}{-4}$ & 527 & 232 & 106 \\
\noalign{\smallskip}\hline
\end{tabular}
\end{table}

\vspace{-4em}

\begin{figure}[H]
	\begin{center}
		\includegraphics[width=0.9\textwidth]{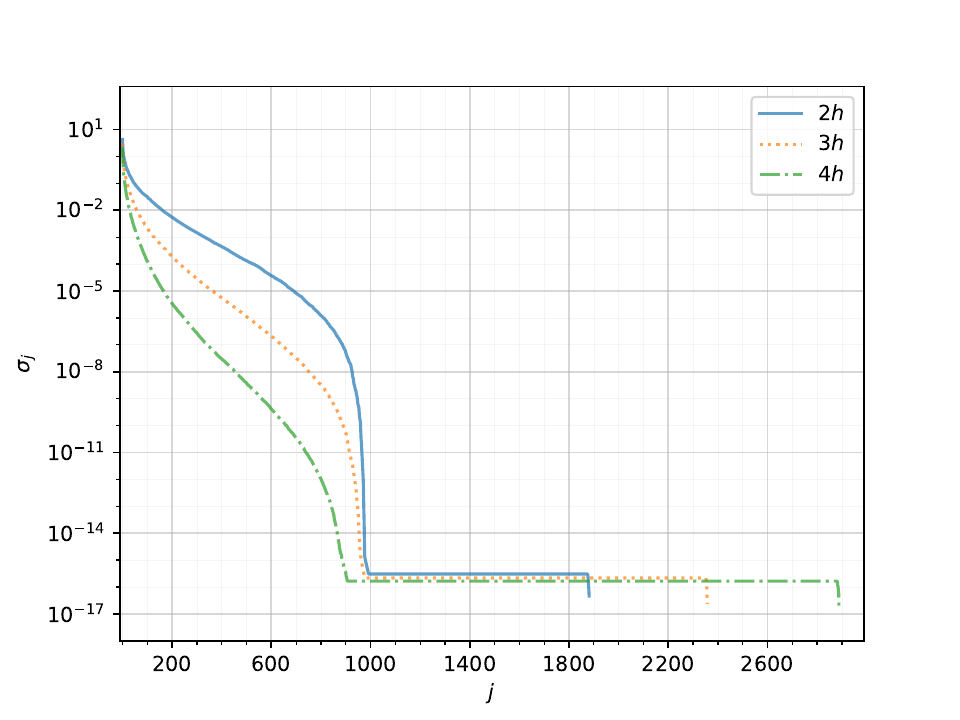}
	\end{center}
	\caption{The sorted singular values of weighted $\bm Z_i$ with three different subdomain extension parameters $r$ plotted on a logarithmic $y$-axis using a second-degree polynomial basis. The original subdomain had $\numprint{3045}$ DOFs and the extensions ranged from $\numprint{12159}-\numprint{29069}$ DOFs. The subdomain diameter is roughly doubled for $r=4h$. Larger extensions produce faster spectral decay as high frequency modes diminish faster. Number of singular vectors for different extensions and tolerances are presented in Table \ref{tbl:localcutoffs}.    
    }\label{fig:zspectrum}
\end{figure}

It is clear that only a minor subset of the extension boundary degrees-of-freedom are needed to ensure a good local estimate. For $r=4h$ and $\epsilon=\expnumber{1}{-3}$, only some dozens of singular vectors suffice, a two magnitude reduction compared to the original degrees-of-freedom $\numprint{3045}$ even in this challenging cube case. The fewer degrees-of-freedom required, the better the reduction and hence smaller matrices in \eqref{eq:nitschesystemred}. 

Increasing the extension parameter $r$ leads to smaller reduced bases, but at an increasing computational cost to the local problems. 
Further, the decay of singular values is faster when raising the polynomial order from $p=1$ to $p=2$ for fixed $r$, compare to results for $p=1$ and same $r$ in \cite[Figure C.9]{gustafsson2024} . This is to be expected as there are more degrees-of-freedom per element for higher polynomial finite element bases, hence, better approximations relative to the degrees-of-freedom on the extension boundary.

It can be useful to have as few subdomains as possible to reduce the interface dimensionality -- and thus the size of the Schur complement system -- to the degree local resources have enough memory. This can result in faster convergence with the conjugate gradient method. However, if the extension parameter is untouched, the extension decreases in relation to subdomain diameter because the subdomains are now larger. This results in less reduction and counters the decrease in number of nonzero elements due to fewer local reduced bases. 
These dynamics are presented in Table \ref{tbl:subdomains}, which presents varying the number of subdomains for case 4 in Table \ref{tbl:convergence}. 

\begin{table}[h]
\centering
\caption{Decreasing the number of subdomains (within a reasonable range) decreases the dimensionality of the Schur complement system but increases the number of nonzeroes, and vice versa. Results for case $4$ in Table \ref{tbl:convergence} when varying the number of subdomains $n$. Even when the extension becomes smaller compared to subdomain diameter when $n$ decreases, the number of nonzeroes remains practically unchanged as there are fewer local reduced bases.} 
\label{tbl:subdomains}
\begin{tabular}{ccccccccc}
\hline\noalign{\smallskip}
$\dim(V)$ & $n$ & $\dim(\widetilde{\bm S})$ & $\dim(\bm \Lambda)$ & nnz($\widetilde{\bm S})$ & Iter$_{CG}$ & $\kappa(\widetilde{\bm S})$\\
\noalign{\smallskip}\hline\noalign{\smallskip}
\numprint{1601613} & \numprint{700}& \numprint{316804} &\numprint{52717} & \numprint{57757585} & 467 & $\expnumber{1.8}{+4}$ \\
\numprint{1601613} & \numprint{350}& \numprint{249742} & \numprint{34789} & \numprint{58171388} & 441 & $\expnumber{1.5}{+4}$ \\
\numprint{1601613} & \numprint{175}& \numprint{182390} & \numprint{21190} & \numprint{58542015} & 392 & $\expnumber{1.2}{+4}$ \\
\noalign{\smallskip}\hline
\end{tabular}
\end{table}

\subsection{Larger example}
\label{sec:scaling}
Cube is a very demanding object for the method. The geometry produces large
interfaces between subdomains for any partition and thus larger optimal reduced
bases. For a more challenging problem, we had a unit cube discretized to
$\numprint{21717639}$ degrees-of-freedom and partitioned into $\numprint{6000}$ subdomains using METIS \cite{karypismetis1997}. The load \eqref{eq:testload} was
solved with the parameters $p=2, \epsilon = \expnumber{1}{-5}, r=4h$ and $\alpha=0.01$.
Mesh preprocessing and solving \eqref{eq:nitschesystem} was done on a laptop main node. The local reduced bases were formed in Google Cloud using machine type \texttt{c2-standard-4} and image \texttt{debian-11-bullseye-v20230411} at spot prices for lower costs. The computational nodes are detailed in Table \ref{tbl:compenv}.

\begin{table}
\centering
\caption{Computational environment for the scaling test.} 
\label{tbl:compenv}
\begin{tabular}{lllcr}
\hline\noalign{\smallskip}
Node & OS & CPU & Threads & RAM \\
\noalign{\smallskip}\hline\noalign{\smallskip}
Main & Ubuntu 22.04 LTS & Intel Core i5-1335U & 12 & 32GB \\
Worker & Debian Bookworm 12.7 & Intel Xeon Gold 6254 & 4 & 6GB  \\
\noalign{\smallskip}\hline
\end{tabular}
\end{table}

The reduced system was solved using the conjugate
gradient method with a diagonal preconditioner. The diagonal preconditioner can be given as a sum of local components, which were created on the worker nodes to avoid the explicit construction of the system. The original over 20 million degrees-of-freedom were reduced to $\tilde k = \dim(\bm \Lambda) = \numprint{1362828}$ degrees-of-freedom and the Schur complement system was $K= \dim(\widetilde{\bm S}) = \numprint{4038252}$-dimensional. The error \eqref{eq:energynormerror} was $\expnumber{7.8}{-5}$, where the mean error for the $\numprint{6000}$ subdomains was $\expnumber{1.0}{-6}$ and maximum $\expnumber{2.3}{-6}$. The error follows the theoretical convergence rate of FEM for second-degree polynomials.

With these resources, the method was memory bound by the main node with 32GB of RAM, and the creation of the reduced Schur complement system \eqref{eq:nitschesystemred} utilized swap memory momentarily. Hence, using, e.g., a large memory main node from the cloud, the method could be straightforwardly scaled further. Moreover, the method is most applicable to complex geometries that admit partitions with small subdomain interfaces. These result in smaller Schur complement systems and more reduction. This loosens requirements for the computing environment relative to degrees-of-freedom of the original system and thus allows for solving significantly larger problems than e.g. the 20 million degrees-of-freedom cube presented here.

\subsection{Engineering model problem}
\label{sec:engineering}

\begin{figure}[H]
	\begin{center}
		\includegraphics[width=0.45\textwidth]{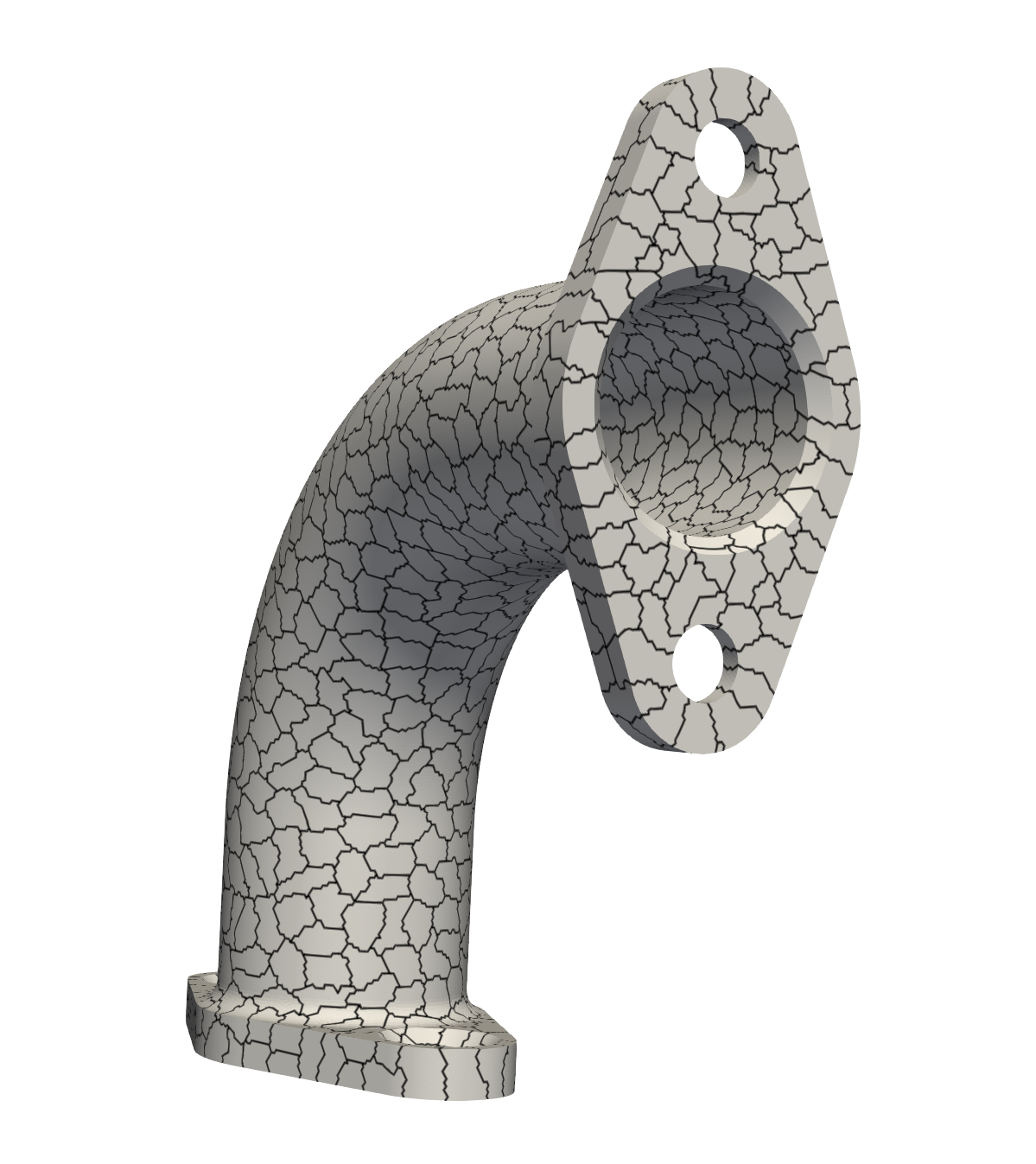}
	\end{center}
	\caption{The pipe geometry discretized into $\numprint{345821}$ nodes and 800 subdomains.}\label{fig:pipe}
\end{figure}
As a more practical example, we considered a curved pipe geometry from \cite{ledouxmamboprojectnodate}, see Figure \ref{fig:pipe}. This kind of a 2.5-dimensional problem is more suitable to our methodology, as the dimension reduction is dependent on the size of the subdomain interfaces. Given the shallow depth of the pipe, it can be partitioned into subdomains with relatively small subdomain interfaces, and hence, lower dimensional trace spaces. This allows for a substantial reduction in the number of degrees-of-freedom. 

We used the load $f=1$ and discretized the pipe into $\numprint{345821}$ nodes. With $p=2$, the original system had $\numprint{2550753}$ degrees-of-freedom. The other parameters were $\epsilon=\expnumber{1}{-4}, n=800, r=4h$ and $\alpha=0.01$. The original system was reduced to $\numprint{29363}$ degrees-of-freedom, a $98.8\%$ reduction, while the trace variable and hence the system \eqref{eq:nitschesystemred} were $\numprint{271795}$-dimensional. The relative reduction error compared to the conforming finite element solution was $\expnumber{7.4}{-4}$.

\section{Conclusions}
\label{sec:conclusions}

Considering a simple model problem, we presented its domain decomposition formulation with Lagrange multipliers and the corresponding FEM approximation based on a hybrid Nitsche formulation of arbitrary polynomial degree. Next, we modified the hybrid Nitsche formulation using the local model order reduction scheme introduced in \cite{gustafsson2024}.
We proved polynomial convergence with respect to the mesh parameter and linear convergence with respect to user-specified local error tolerance $\epsilon$. 
This improves upon the existing first-order estimate \cite{gustafsson2024}. Finally, we presented matrix implementation details and validated the theoretical results with numerical tests. The methodology shows promise for large-scale computing especially for challenging geometries, but development of a specialized preconditioner for the resulting Schur complement system is necessary.

\clearpage

\appendix

\section{Hybrid Nitsche matrix form}
\label{apx:matrixelements}

Let $\Set{\varphi_j^i}_{j=1}^{m_i}$ be a basis for $V_{h,i}$ and $\Set{\xi_j}_{j=1}^K$ be a piecewise basis for $V_{h,0}$. The matrices $\bm A_i\in \R^{m_i\times m_i}, \bm B_i\in \R^{m_i\times K}, i=1,\dots,n$ and $\bm C\in \R^{K\times K}$ have elements:
\begin{align*}
	(\bm A_i)_{jk} & = \int_{\Omega_i} \nabla \varphi_j^i \cdot \nabla \varphi_k^i\, dx - \int_{\partial\Omega_i} \frac{\partial \varphi_j^i}{\partial n} \varphi_k^i + \varphi_j^i\frac{\partial\varphi_k^i}{\partial n} - \frac{1}{\alpha h_i}\varphi_j^i \varphi_j^k\, ds, \\
	(\bm B_i)_{jk} & = \int_{\partial\Omega_i} \frac{\partial\varphi_j^i}{\partial n}\xi_k - \frac{1}{\alpha h_i} \varphi_j^i \xi_k\, ds, \\
	\label{eq:celem}
	\bm C_{kl} & = \int_{\Gamma} \frac{1}{\alpha h_i}\xi_k\xi_l + \frac{1}{\alpha h_j} \xi_l\xi_k\, ds, \quad (\mathrm{supp}(\xi_k)\cap\mathrm{supp}(\xi_l))\subset (\partial\Omega_i\cap\partial\Omega_j), \\
	(\bm f_i)_j & = \int_{\Omega_i} f\varphi_j^i\, dx.
\end{align*}
Further, the solution can be written as follows $u_{h} = [\bm\beta_1\cdot \bm\varphi^1, \dots, \bm\beta_n\cdot \bm\varphi^n, \bm\beta_0\cdot \bm\xi].$

\section*{Competing interests and funding}

This work was supported by the Research Council of Finland (Flagship of Advanced Mathematics for Sensing Imaging and Modelling grant 359181) and the Portuguese government through FCT (Funda\c c\~ao para a Ci\^encia e a Tecnologia), I.P., under the project UIDB/04459/2025.

\bibliography{references}
\bibliographystyle{plain}

\end{document}